%% file: Kernel.tex
\documentclass[review]{elsarticle}
\usepackage{anysize}
\usepackage{amsmath}  
\usepackage[english]{babel}
\usepackage[utf8]{inputenc}
\usepackage{amsthm} 
\usepackage{amsfonts}  
\usepackage{graphicx}   
\usepackage{epstopdf}
\usepackage{verbatim}   
\usepackage{color}      
\usepackage{subfigure}  
\usepackage{hyperref}
\numberwithin{equation}{section}
\usepackage[titletoc,title]{appendix}
\usepackage{mathtools}
\usepackage{dsfont}
\usepackage{pdfpages}
\input{Commands}

\begin{document}
\begin{frontmatter}
	
	\title{Time-fractional Birth and Death Processes}	
	\author[mysecondaryaddress]{Jorge Littin Curinao}
	\ead{jlittin@ucn.cl}	
	\address[mymainaddress]{Universidad Cat\'olica del Norte, Departamento de Matem\'aticas, Avenida Angamos 0610, Antofagasta - Chile}
\begin{abstract}

In this article, we provide different representations for a time-fractional birth and death process $N_{\frexp}(t)$, whose transition probabilities  $\defpijt$ are governed by a time-fractional system of differential equations. More specifically, we present two equivalent characterizations for its trajectories: the first one as a time-changed classic birth and death process, whereas the second one is a Markov renewal process. Also, we provide results for the asymptotic behavior of the process conditioned not to be killed. The most important is that the concept of quasi-limiting distribution and quasi-stationary distribution do not coincide, which is a consequence of the long-memory nature of the process. As an application example, we revisit the linear case to show the consequences of our main theorems.
	\end{abstract}	
\end{frontmatter}
\section{Introduction}
The birth and death processes have been extensively studied in different areas of both probability theory and its applications in population models, epidemiology, queuing theory, and engineering, to name a few. Two fundamental aspects related to its analysis are the representation of the transition probabilities that model the evolution of the system and the asymptotic behavior after a long time.

Since many processes exhibit the phenomenon of long memory, a Markov process seems no appropriate at all, so that fractional models appear to be more precise. Time-fractional models in the context of anomalous diffusion have been studied previously by Orsingher \cite{Orsingher2004,orsingher2009}, where the time-fractional telegraph equation and a fractional diffusion equation where analyzed respectively. Previous results for fractional birth and death processes can be found in the articles of Orsingher \cite{LinearBDP2011} for the linear case, Meerschaert \cite{meerschaert2011} for the fractional Poisson process and Jumarie \cite{JUMAR2010} for a pure birth and death process with multiple births. Surprisingly, none of them provide representations for an arbitrary time-fractional birth and death process. This means that results concerning the asymptotic behavior are no available in the fractional case. 

For Markov processes, the study of the number of survival after a long time started with the early work of Kolmogorov in 1938. Later in 1947, Yaglom \cite{Yaglom} showed that the limit behavior of sub-critical branching processes conditioned to survival was given by a proper distribution. In 1965, Darroch \& Seneta \cite{Darroch} started the study of Quasi Stationary Distributions (qsd) on finite state irreducible Markov Chain, whereas Seneta and Vere Jones \cite{Vere} on 1966 did it for Markov Chains with countable states. A  very important publication was done by Van Doorn in 1991 \cite{Doorn91}, which states a criteria to determine the existence and uniqueness of qsd for birth and death chains. More recent results about the existence and uniqueness of qsd can be revised on \cite{VANDOORN20122400, VANDOORN20131}.

For diffusion processes on the half-line, the first work is due to Mandl \cite{Mandl1961}, who studied the existence of a qsd on the half-line for $+\infty$ being a natural boundary accordingly to Feller's classification. In subsequent works, many some result of existence of qsd and limit laws  for one dimensional diffusions killed at 0 are provided by Ferrari \cite{ferrari1995}, Collet Mart\'inez San Martín \cite{mpsm}   and Mart\'inez San Mart\'in \cite{cmsm,smjsm}.

Most of these works are based on studying the spectral decomposition of the infinitesimal operator associated with the process. Applying similar ideas, we can study the asymptotic behavior of time-fractional models, which is precisely one of the main objectives of this article.

This article is organized as follows: in section \ref{PRESMOD}, we present the model description. More specifically, we introduce the system of time-fractional equations that governs the transition probabilities. In section \ref{CARPRO}, two equivalent characterizations are shown: the first one is a time-changed birth and death process, whereas the second one is a Markov Renewal process. In section \ref{REPSP}, we follow a different approach based on a spectral representation of the transition probabilities to study the quasi limiting behavior of the process conditioned not to be killed. In section \ref{QSDB}, we study the concept of quasi-stationary distributions proving that the quasi limiting distribution and quasi-stationary distribution are not the same. Finally, in section \ref{EXAM}, we apply the main theorems to the linear model.

\section{Model formulation}\label{PRESMOD}
We call $N_{\frexp}(t)$, $t \geq 0$, to the fractional birth and death process killed at zero. The transition probabilities denoted by
\begin{equation}
\pijt=\defpijt
\end{equation}
are governed by the time-fractional  system of differential equations  (commonly called system of backward equations)
\begin{eqnarray}\label{BACKW}
\derfra{\pijt}&=&\BACK,\hspace{1cm} j \geq 1,\\
\pojt&=&0.
\end{eqnarray}
As usual, the values $\lambda_i>0, \mu_i>0$ (with the convention $\mu_0=\lambda_0=0$) are  the birth rates and the death rates respectively, whereas the parameter $\frexp \in (0,1]$ determines the order of the  Caputo-Riemann-Liouville fractional  operator $\derfra{(\cdot)}$, defined as
\begin{eqnarray}
\derFr{\frexp}(t)&=&\Fracdef{\frexp},\hspace{2cm}\frexp \in (0,1),\\
\derFr{1}(t)&=&f^{\prime}(t), \hspace{4.9cm}\frexp =1.
\end{eqnarray}
In particular, when $\frexp=1$ the operator is just the derivative and $N_1(t)$ is the classical birth and death process. The matrix formulation for the  system of equations \eqref{BACKW} is
\begin{equation}
\derfrM{\frexp}=QP(t),
\end{equation}
where  $P(t)$ is the matrix with coefficients $\pijt$, $i \geq 0$, $j \geq 0$ and the matrix $Q$ is defined as
\begin{eqnarray}\label{qij}
q_{i,j}= \left\{ \begin{array}{ccc}
-(\lambda_i+\mu_i) &   if  & i=j, \\
\\ \mu_i &  if & i=j-1, \\
\\ \lambda_i &  if  & i=j+1 .
\end{array}
\right. 
\end{eqnarray}
The initial condition is the  Kronecker delta $p_{i,j,\frexp}(0)=\delta_{i,j}$. 

\section{Two equivalent characterizations}\label{CARPRO}
In this section, we introduce the representation of the process $N_\frexp (t)$ as a usual birth and death chain changed in time, this is  $N_\frexp (t)=N_1(\Et{t})$, where $\Et{t}$ is the inverse of a stable subordinator. Also, we get a representation as a  Markov Renewal Process in the general case. To make this work self contained, we introduce first some general facts concerning stable subordinators and its inverse.

\subsection{The stable subordinator and its inverse}
A  subordinator $\Dt{t}$, $t \geq 0$ is a one-dimensional Levy Process such that their trajectories are non decreasing with probability $1$. In particular, we say that a subordinator is stable when the  Laplace transform of the $\Dt{1}$ satisfy
\begin{equation}\label{LaplT}
E[e^{-s \Dt{t}}]=e^{-ts^\frexp}.
\end{equation}
Associated to a subordinator $\Dt{t}$, we define the inverse process  $\Et{t}$, $t \geq 0$ as follows
\begin{equation}\label{ET}
\Et{t}=\inf \{ r \geq 0: \Dt{r}>t \}.
\end{equation}
The process $\Et{t}$ denotes the  first time that  $\Dt{t}$ exceeds a level $t>0$. It is clear that the trajectories of the process $\Et t$ are  non-decreasing and continuous. From the equation \eqref{ET}, we can deduce that the finite dimensional distributions of $\Dt{t}$ and $\Et{t}$ satisfy the identity
\begin{equation}\label{ETDT}
P\left[ \Et{t_i} > x_i, 1 \leq i \leq n\right]=P\left[ \Dt{x_i} < t_i, 1 \leq i \leq n\right].
\end{equation}
The equation \eqref{LaplT}  directly implies that the process $\Dt{t}$ is self similar of index $1/\frexp$, {\it i.e}
\begin{equation}\label{SSim}
P\left[ \Dt{cx_i} < t_i, 1 \leq i \leq n\right]=P\left[ c^{1/\frexp}\Dt{x_i} < t_i, 1 \leq i \leq n\right].
\end{equation}
Moreover, from equations \eqref{ETDT} and \eqref{SSim} we have that the process $\Et{t}$ is self similar of index $\frexp$
\begin{equation}
P\left[ \Et{c t_i} > x_i, 1 \leq i \leq n\right]=P\left[ c^\frexp \Et{t_i} > x_i, 1 \leq i \leq n\right].
\end{equation}
For all $t>0$, the distribution of  $D(t)$ has only a density component, here denoted by $g_\frexp(\cdot,t)$. In the same way, the inverse $\Et t$ has only a density component $h_\frexp(\cdot,t)$  satisfying
\begin{equation}
h_\frexp(u,t)=\frac{t}{\frexp}{u^{-1-1/\frexp}}g_\frexp(t u ^{-\frac{1}{\frexp}},t).
\end{equation}
Concerning the Laplace transform of $h_\frexp(u,t)$ we have the identities
\begin{eqnarray}\label{LEBSTI}
\int_{0}^{\infty}  e^{-s t} \dPEb dt&=&s^{\frexp-1}e^{-us^\frexp},\\
 \int_{0}^{\infty}  e^{-st} \dPEb du&=&\ML{-st^\frexp},
\end{eqnarray}
where $\ML{\cdot}$ is the Mittag-Leffler function formally defined as
\begin{equation}
\ML{z}=\defML{z},\hspace{1cm} z \in \C.
\end{equation}
Finally, we emphasize that the  increments of the process  $\Et{t}$  are dependent and non-stationary (see corollary 3.3 and 3.4 from  \cite{meerschaert2004} for a formal proof ot this fact).
\subsection{The time changed process}
\begin{theorem}\label{ChanTime}
For all $\frexp \in (0,1)$, the stochastic process $N_\frexp(t)$, $t \geq 0$ admits a representation (in the sense ot the finite-dimensional distributions) into the form 
	\begin{equation}
	N_\frexp(t)=N_1(\Et{t}),
	\end{equation}
  where $N_1(t)$ is an usual birth  and death process and $\Et{t}$ is the inverse of a stable subordinator  independent of $N_1(t)$.
\end{theorem}
\begin{proof}
Given fixed $i > 1$ and  $j \geq 0$, for all $t>0$ we have
	\begin{eqnarray}
	P_i[N_1(\Et{t})=j]&=&\int_{0}^{\infty}  p_{i,j,1}(u) \dPEb dt\\
	&=&E[p_{i,j,1}(\Et t)].
	\end{eqnarray}
It suffices to prove that $E[p_{i,j,1}(\Et t)]$ is the solution to the system of equations  \eqref{BACKW}. We claim that
	\begin{equation}
	\Psi_{i,j}(s)=\int_{0}^{\infty} e^{-s t}E[p_{i,j,1}(\Et t)]dt
	\end{equation}  
satisfies the identity $$\Psi_{i,j}(s)=s^{\frexp-1}\Phi_{i,j}(s^\alpha),$$
	 being $\Phi_{i,j}(s)=\int_{0}^{\infty} e^{-s t}p_{i,j,1}(t)dt$ the Laplace transform of $p_{i,j,1}(t)$. In fact
	\begin{eqnarray}
	\Psi_{i,j}(s)&=&\int_{0}^{\infty}e^{-s t}\left( \int_{0}^{\infty} p_{i,j,1}(u) \dPE\right)dt\nonumber\\
	&=&\int_{0}^{\infty}p_{i,j,1}(u) \left( \int_{0}^{\infty}  e^{-s t} \dPEb dt \right)du\nonumber\\
	&=&s^{\frexp-1}\int_{0}^{\infty}p_{i,j,1}(u)  e^{-s^\frexp u}du\nonumber\\
	&=&s^{\frexp-1}\Phi_{i,j}(s^\alpha)\label{LAPID}.
	\end{eqnarray}
On the other hand, it is well known that $\Phi_{i,j}(s^\alpha)$ defined above, satisfy the system of equations
	\begin{equation}\label{LAPTRPRO}
	s^\frexp \Phi_{i,j}(s^\alpha)-\delta_{i,j}=\BACKLAP.
	\end{equation}
	By combining the equations \eqref{LAPID} and \eqref{LAPTRPRO} we get 
	\begin{equation}
	\frac{s \Psi_{i,j}(s)-\delta_{i,j}}{s^{1-\frexp}}=\BACKLAPB,
	\end{equation}
	 and finally, by taking the inverse transform, we deduce directly  that $P_i[N_1(\Et{t})=j]$ is the solution to the system \eqref{BACKW}, concluding the proof.
\end{proof}
\begin{remark}
  The trajectories of $\Et{t}$ are non-decreasing, so that the theorem \ref{ChanTime} can be used to obtain the finite dimensional distributions of $N_\frexp(t)$
	\begin{eqnarray}
	P_i[N_\frexp(t_l)=j_l; 1 \leq l \leq n]&=&P_i[N_1(\Et{t_l})=j_l; 1 \leq l \leq n]\\
	&=&\prod_{l=0}^{n-1}P_{j_l}[N_1(\Et{t_{l+1}}-\Et{t_l})=j_{l+1}],\label{finitedim}
	\end{eqnarray}
with the convention $j_0=i$, $t_0=0$.	
\end{remark}

\subsection{Markov Renewal Process}

\begin{definition}
Let $\{X_n\}_{n \geq 0}$ be a Markov chain with states in  $\N_0$ and let $\{\cS_n\}_{n \geq 0}$ be a sequence of random times satisfying  $\cS_0=0$, $\cS_n < \cS_{n+1}$ for all $n \geq 0$. The stochastic processes $\{X_n,\cS_n\}_{n \geq 0}$ is called a Markov Renewal process  with space state $\N_0=\{0,1,2,3,\cdots\}$ if the identity
	\begin{eqnarray*}
		P[X_{n+1}=j,\Tau_{n+1} \leq t|(X_i,\cS_i), 1 \leq i \leq n]&=&P[X_{n+1}=j,\Tau_{n+1} \leq t|X_n]
	\end{eqnarray*}
follows for all $j \in \N_0$, $n \geq 0$ and $t \geq 0$. Here $\Tau_{n+1}=\cS_{n+1}-\cS_n$, $n \geq 0$ are called the inter arrival times.
\end{definition}
Connected to a Markov Renewal Process we consider the transition probabilities
\begin{equation}
p_{i,j}=P[X_{n+1}=j|X_n=i]
\end{equation}
and the kernel
\begin{equation}
Q_{i,j}(t)=P[X_{n+1}=j,\Tau_{n+1} \leq t|X_n].
\end{equation}
The transition probabilities are recovered in the limit $t \rightarrow \infty$
\begin{eqnarray}
\lim_{t \rightarrow \infty } Q_{i,j}(t)=\lim_{t \rightarrow \infty } P[X_{n+1}=j,\Tau_{n+1} \leq t|X_n]
&=& p_{i,j}.
\end{eqnarray}
By introducing the notation
\begin{equation}
G_{i,j}(t)=\frac{Q_{i,j}(t)}{p_{i,j}},
\end{equation}
from a direct computation we get
\begin{equation}
P[\Tau_{n} \leq t|X_{n-1}=i,X_{n}=j]=G_{i,j}(t)
\end{equation}
and more generally for all finite collections of times $0<t_1<t_2<\cdots <t_n$ 
\begin{equation}\label{findimdis}
P[\Tau_{i} \leq t_i; 1 \leq i \leq n|X_0,X_1,\cdots X_n]=\prod_{i=1}^{n}G_{X_{i-1},X_{i}}(t_{i}).
\end{equation}
The  equation \ref{findimdis} implies that the inter arrival times $\Tau_i$, $ i \geq 0$ conditioned to the chain $X_n$, $n\geq 0$ are independent with distribution $G_{X_i,X_{i+1}}$. It is well known that a  Markov Renewal Process is characterized by  $G_{X_i,X_{i+1}}(t)$ and the transition probabilities $p_{i,j}$. In addition, this is a Markov process if and only if $G_{X_i,X_{i+1}}(t)=1-e^{-r(X_i,X_{i+1})t}$, for some positive rate $r(X_i,X_{i+1})$. A more detailed review of these results can be found in \cite{cinlar2013}. The following theorem states that $N_\frexp(t)$ is a Markov Renewal Process.

\begin{theorem}\label{SMP}
For all $\frexp \in (0,1)$ the process $N_\frexp(t)$ admits a representation  into the form
	\begin{equation}
	N_\frexp(t)=\sum_{n \geq 0} X_n \mathds{1}_{\cS_{k,\frexp} \leq t<\cS_{k+1,\frexp}},
	\end{equation}
	where $(X_n,\cS_{n,\frexp})_{n \geq 0}$ is a Markov Renewal process. The transition probabilities are
	\begin{equation}\label{pijd}
	P[X_{k+1}=i+1|X_k=i]=\frac{\lambda_{i}}{\lambda_i+\mu_i}, \hspace{1cm} 	P[X_{k+1}=i-1|X_k=i]=\frac{\mu_i}{\lambda_i+\mu_i}
	\end{equation}
	and the inter-arrival times 
	\begin{equation}
	\Tau_{k,\frexp}=\cS_{k+1,\frexp}-\cS_{k,\frexp},
	\end{equation}
	conditioned to $\{X_i\}_{1 \leq i \leq k}$ are independent and they follow a  Mittag-Leffler distribution with parameter $\lambda_{i}+\mu_{i}$, i.e.
	\begin{equation}
	G_{i,j}(t)=1-\ML{-(\rate{i}) t^\frexp}.
	\end{equation}
\end{theorem}
\begin{proof}
 We first remark  that for $\frexp=1$ the theorem is valid. More precisely, conditioned to $\{X_i\}_{1 \leq i \leq k}$, the inter arrival times independent are  exponentially distributed with parameter $\rate{X_n}$. When $\frexp \in (0,1)$, from Theorem \ref{ChanTime} we know that $N_\frexp(t)=N_{1}(\Et{t})$ is a time changed process. Thus
	\begin{eqnarray*}
		P[X_{n+1}=j,\Tau_{n+1,\frexp} \leq t|(X_i,\cS_i), 1 \leq i \leq n]&=&\int_{0}^{\infty} P[X_{n+1}=j,\Tau_{n+1,1} \leq u|(X_i,\cS_i), 1 \leq i \leq n] \dPE\\
		&=&\int_{0}^{\infty}	P[X_{n+1}=j,\Tau_{n+1,1} \leq u|X_n] \dPE\\
		&=&P[X_{n+1}=j,\Tau_{n+1,\frexp} \leq t|X_n].
	\end{eqnarray*}	
The transition probabilities are obtained by using the fact that $\lim_{t \rightarrow \infty}\Et{t}=\infty$ with probability $1$
	\begin{eqnarray}
	p_{i,j}^\frexp&=&\lim_{t \rightarrow \infty } P[X_{n+1}=j,\Tau_{n+1,\frexp} \leq t|X_n]\\
	&=&\lim_{t \rightarrow \infty } P[X_{n+1}=j,\Tau_{n+1,1} \leq \Et{t}|X_n]\\
	&=&\lim_{u \rightarrow \infty } P[X_{n+1}=j,\Tau_{n+1,1} \leq u|X_n]\\
	&=& p_{i,j}.
	\end{eqnarray}
The distribution of the inter-arrival times is deduced inductively. Since $\SUMT{\frexp}=\sum_{i=0}^{n-1}\Tau_{i,\frexp}$, when $n=1$ we have
	\begin{eqnarray}
	P_i[\Tau_{0,\frexp}>t]&=&P[\cS_{1,\frexp}>\Et{t}]\\
	&=&\int_{0}^{\infty}P[\cS_{1,\frexp}>u]\dPE\\
	&=&\int_{0}^{\infty}e^{-(\rate{X_0}) u }\dPE\\
	&=&E[e^{-(\rate{X_0}) \Et{t} }]\\
	&=&\ML {-(\rate{X_0}) t^\frexp}.
	\end{eqnarray}
 For $n \geq 2$ we define
	\begin{equation}\label{accumprob}
	\varphi_{n,\frexp}(s)=\int_{0}^{\infty}	e^{-st}P\left[\SUMT{\frexp}\leq t\right]dt,
	\end{equation}
	which is the Laplace transform of $P\left[\SUMT{\frexp}\leq t\right]$. Since $N_\frexp(t)$ is a time changed process, it is fulfilled
	\begin{eqnarray}
	\varphi_{n,\frexp}(s)
	&=&\int_{0}^{\infty}	e^{-st}P\left[\SUMT{1}\leq \Et{t}\right]dt\\
	&=&\int_{0}^{\infty}	e^{-st}\left(\int_{0}^{\infty}P\left[\SUMT{1}\leq u\right]\dPE \right) dt,\label{LAPLTRa}\\
	\end{eqnarray}
	by using  the Fubini's theorem and recalling the identity 
	\begin{eqnarray}\label{TransLA}
	\int_{0}^{\infty}e^{-st} h_\frexp(u,t) dt=s^{\frexp-1}e^{-us^\frexp}
	\end{eqnarray}
	we get
	\begin{eqnarray}
	\varphi_{n,\frexp}(s)
	&=&\int_{0}^{\infty}	e^{-st}\left(\int_{0}^{\infty}P\left[\SUMT{1}\leq u\right]\dPE\right) dt\\
	&=&\int_{0}^{\infty}P\left[\SUMT{1}\leq u\right]	\left(\int_{0}^{\infty}e^{-st} \dPEb dt \right) du\\
	&=& \int_{0}^{\infty}P\left[\SUMT{1}\leq u\right]	s^{\frexp-1}e^{- s^\frexp u}du\label{PREINTPAR}\\
	&=&	s^{\frexp-1} \varphi_{n,1}(s^\frexp).
	\end{eqnarray}

	When $\frexp=1$ the inter-arrival times are independent and exponentially distributed, so	
	\begin{equation}
	\varphi_{n,1}(s^\frexp)=\frac{1}{s^\frexp}\prod_{i=0}^{n-1}\frac{\lambda_{X_i}+\mu_{X_i}}{s^ \frexp+\lambda_{X_i}+\mu_{X_i}}	,
	\end{equation}
	leading to the formula
	\begin{equation}\label{formula}
	\varphi_{n,\frexp}(s)=\frac{1}{s}\prod_{i=0}^{n-1}\frac{\lambda_{X_i}+\mu_{X_i}}{s^ \frexp+\lambda_{X_i}+\mu_{X_i}}.	
	\end{equation}	

	  The equation \eqref{formula} directly implies that for all $n \geq 2$, the random variable $\cS_{n,\frexp}$ is the sum of $n-1$ independent random variables with a Mittag- Leffler distribution. Consequently, conditioned to $X_n$, the inter-arrival times  $\Tau_{n,\frexp}=\cS_{n+1,\frexp}-\cS_{n,\frexp}$ are independent and they satisfies
	\begin{equation}
	P[\Tau_{n,\frexp}>t|X_n]=\ML{-(\lambda_{X_n}+\mu_{X_n})t^\frexp},
	\end{equation}
	concluding the proof.
\end{proof}
\section{The spectral representation of the transition probabilities}\label{REPSP}
\subsection{Preliminaries}
For all $i \geq 0$ fixed, we denote by
\begin{equation}
T_{i,\frexp}= \inf \{ t>0: N_{\frexp}(t)=i\},
\end{equation}
the first time that the process $N_{\frexp}(t)$ attains the state $i$. In particular, for $i=0$, we say that  $T_{0,\frexp}$ is the absorption time  or the extinction time of the process. For the sake of convenience, for $\frexp=1$ we write $T_0$ instead of $T_{0,1}$.

For $n \geq 1$ we define the coefficients
\begin{eqnarray}\label{pn}
\pi_1&=&1,\\ 
\pi_n&=&\prod_{i=1}^{n-1}\frac{\lambda_i}{\mu_{i+1}}, \hspace{0.5cm}n \geq 2.
\end{eqnarray}
Note that these coefficients satisfy the identity $\pi_{n+1}/\pi_{n}=\lambda_{n}/\mu_{n+1}$. This implies that the process is reversible with respect to the measure $\pi$, this is
\begin{equation}
\pi_{i}q_{i,j}=\pi_{j}q_{j,i}	\textrm{ for all } i,j \geq 1.
\end{equation}
The following series are essential to describe some properties of the process
\begin{equation}
\begin{array}{ll}
A=\displaystyle \sum_{i\geq 1}(\lambda_{i}\pi_{i})^{-1},& B=\displaystyle \sum_{i\geq 1}\pi_{i},\\C=\displaystyle \sum_{i\geq 1}(\lambda_{i}\pi_{i})^{-1}\sum_{j=1}^{i}\pi_{j},
&D=\displaystyle \sum_{i\geq 2}(\mu_{i}\pi_{i})^{-1}\sum_{j\geq i}\pi_{j}.
\end{array}
\end{equation}
When $\frexp=1$, some well-known results are (see for instance Chapter 5 of \cite{servet2012qsd})
\begin{itemize}
	\item[1)] The process is almost surely absorbed at zero, i.e. $P[T_0< \infty] =1$, if and only if $A=\infty$.
	\item[2)] The absorption time has a finite mean, i.e. $E_i[T_0]<\infty$ if and only if $B<\infty$.
	\item[3)] The process comes from infinity, i.e. $\sup_{i \geq 1}E_i[T_0]<\infty$ if and only if $D<\infty$. 
\end{itemize}

\subsection{Main Results}
Given a fixed $\theta > 0$, we define recursively the sequence of polynomials   
\begin{eqnarray}\label{recursive}
-\coef \Qt{i}&=&\mu_i \Qt{i-1}-(\lambda_i+\mu_i)\Qt{i}+\lambda_i \Qt{i+1},\hspace{0.5cm}i \geq 1,\\
&&\Qt{0}=0, \hspace{.5cm}\Qt{1}=1.
\end{eqnarray}
These polynomials satisfy the orthogonality condition
\begin{equation}
\pi_j \int_{ \theta^\star}^\infty\Qt{j}\Qt{k}d\Gamma(\coef)=\delta_{j,k},
\end{equation}
where $\Gamma$ is a  probability measure supported in $[\coef^\star,\infty)$, for some $\coef^\star \geq 0$ and $\delta_{j,k}$ is the Kronecker delta.

It is well known that when $\frexp=1$, the spectral representation of the transition probabilities is 
\begin{equation}\label{transi0}
\pijt=\pi_j \int_{ \theta^\star}^\infty e^{-t \coef}\Qt{i}\Qt{j}d\Gamma(\coef).
\end{equation}
The following theorem generalizes the representation \ref{transi0} to the complete case $\frexp \in (0,1]$.

\begin{theorem}\label{specrep}
The solution to the system of equations \eqref{BACKW} can be written as
\begin{equation}\label{solpkj}
\pijt=\pi_j \int_{ \theta^\star}^\infty\ML{-\coef t^\frexp}\Qt{i}\Qt{j}d\Gamma(\coef), \hspace{0.5cm} i \geq 1, j \geq 0,
\end{equation}
where $\ML{\cdot}$ is the Mittag-Leffler function with parameter $\frexp$.
\end{theorem}

\begin{proof}
From Theorem \ref{ChanTime},we know  that  $\pijt=p_{i,j,1}(\Et t)$ for all $i \geq 1$, $j \geq 0$. By using the equation \eqref{transi0}, we have for $j \geq 1$
\begin{eqnarray}
\pijt&=&P_i[N_1(\Et{t})=j]\\
&=&\pi_j \int_{0}^{\infty}P_i[N_1(u) =j]\dPE \\
&=&\pi_j \int_{0}^{\infty} \left(\int_{ \coef^\star}^{\infty} e^{-\coef u}\Qt{i}\Qt{j}d\Gamma(\coef)\right)\dPE.
\end{eqnarray}
By using the Fubini's Theorem 
\begin{equation*}
\pijt=\pi_j \int_{\coef^\star}^{\infty}\left(\int_{0}^{\infty} e^{-\coef u}\dPE\right)\Qt{i}\Qt{j}d\Gamma(\coef)
\end{equation*}
and recalling the identity
\begin{equation}\label{dPE}
\ML{-\coef t^\frexp}=\int_{0}^{\infty} e^{-\coef u}\dPE
\end{equation}
we deduce
\begin{equation}\label{TRANPROB}
\pijt=\pi_j \int_{\coef^\star}^{\infty}\ML{-\coef t^\frexp}\Qt{i}\Qt{j}d\Gamma(\coef),
\end{equation}
concluding the proof. 
\end{proof}
\begin{remark}

Since the Mittag-Leffler function satisfies $\derfra{\ML{-\coef t^\frexp}}=-\coef \ML{-\coef t^\frexp}$, we get from theorem \ref{specrep} and the dominated convergence theorem
\begin{equation}\
\derfra{\pijt}=\pi_j \int_{\coef^\star}^{\infty}-\coef{\ML{-\coef t^\frexp}}\Qt{i}\Qt{j}d\Gamma(\coef).
\end{equation}
By applying the recursive formula \eqref{recursive} to $-\coef \Qt{j}$ it follows
			\begin{eqnarray}\label{comp3b}
\derfra{\pijt}&=&\pi_j \int_{ \theta^\star}^\infty\Qt{i} {\ML{-\coef t^\frexp}}\left( \mu_{j} \Qt{j-1}-(\lambda_j+\mu_j)\Qt{j}+\lambda_{j} \Qt{j+1}\right) d\Gamma(\coef).
\end{eqnarray}
After some computations we have that the transition probabilities also satisfies the system of forward equations
\begin{eqnarray}\label{FORW}
\derfra{\pijt}&=&\FORW,\hspace{1cm} j \geq 1.
\end{eqnarray}
In particular, for $j=0$ 
\begin{eqnarray}\label{FRW0}
\derfra{\piot}&=&\mu_1 p_{i,1,\frexp}(t).
\end{eqnarray}
\end{remark}

The following Theorem states the distribution of the absorption time.
\begin{theorem}

For all $i \geq 1$, the probability of  non extinction is
\begin{equation}\label{ABSDIST}
	P_i[T_{0,\frexp}>t]=\mu_1 \int_{ \theta^\star}^\infty\ML{-\coef t^\frexp}\frac{\Qt{i}}{\coef}d\Gamma(\coef).
	\end{equation}	
\end{theorem}
\begin{proof}
By using the equation \eqref{solpkj} for $j=1$
\begin{equation}
	p_{i,1,\frexp}(t)=\pi_i \int_{ \theta^\star}^\infty\ML{-\coef t^\frexp}\Qt{i}d\Gamma(\coef).
	\end{equation}
From equation \eqref{FRW0} and recalling that $\int_{ \theta^\star}^\infty\frac{\Qt{i}}{\coef}d\Gamma(\coef)=1$
	\begin{eqnarray}\label{pi0}
	p_{i,0,\frexp}(t)&=&-\mu_1 \int_{ \theta^\star}^\infty\left(\frac{\ML{-\coef t^\frexp}-1}{\theta}\right)\Qt{i}d\Gamma(\coef)\\
&=&1-\mu_1 \int_{ \theta^\star}^\infty\ML{-\coef t^\frexp}\frac{\Qt{i}}{\coef}d\Gamma(\coef)\label{PRABS}.
\end{eqnarray}	
Since $p_{i,0,\frexp}(t)=1-P_i[T_{0,\frexp}>t]$, the equation \eqref{ABSDIST} follows directly from \eqref{PRABS}.
\end{proof}

\section{Asymptotic Behavior and quasi-limiting distributions}\label{QLDB}
The theorem \eqref{solpkj} allows us to deduce some new results related to the quasi-limiting behavior. We start by defining the integrals
\begin{eqnarray}
C_{i,j,k}=\int_{ \theta^\star}^\infty\frac{\Qt{i}\Qt{j}}{\coef^k}d\Gamma(\coef),\hspace{1cm} i\geq 1, j \geq 1, k \geq 0.
\end{eqnarray}
When $\coef^\star>0$, the coefficients $C_{i,j,k}$ are finite. In fact
\begin{eqnarray}\label{cijk}
C_{i,j,k}&\leq&\frac{1}{(\coef^\star)^k}\left(\int_{ \theta^\star}^\infty{\Qt{i}^2}d\Gamma(\coef)\right)^{1/2} \left(\int_{ \theta^\star}^\infty{\Qt{j}^2}d\Gamma(\coef)\right)^{1/2} <\infty.
\end{eqnarray}
 The following proposition states the asymptotic behavior of both the transition probabilities and the distribution of the absorption time $T_{0,\frexp}$. 

\begin{proposition}\label{limit1}
Assume $\coef^\star>0$. For all $0<\frexp<1$ the following limits are fulfilled
	\begin{eqnarray}
	\lim_{t \rightarrow \infty} t^{\frexp}P_i[T_{0,\frexp}>t]&=&\frac{\mu_1}{\Gamma(1-\frexp)}\int_{ \theta^\star}^\infty\frac{\Qt{i}}{\coef^2}d\Gamma(\coef),\label{ABSTIM}\\
	\lim_{t \rightarrow \infty} t^{\frexp}\pijt&=&\frac{\pi_j}{\Gamma(1-\frexp)}\int_{ \theta^\star}^\infty\frac{\Qt{i}\Qt{j}}{\coef}d\Gamma(\coef)\label{ABSTIM2}.
	\end{eqnarray}
\end{proposition}
\begin{proof}

For all $\coef>0$ the following limit applies
	\begin{eqnarray}
	\lim_{t \rightarrow \infty} t^\frexp{E_{\frexp,1}}(-\coef t^\frexp)&=&\frac{1}{\coef}\frac{1}{\Gamma(1-\frexp)}.
	\end{eqnarray}
	In addition, we know that  $\int_{\theta^\star}^{\infty}\frac{\Qt{j}}{\coef^k}d\Gamma(\coef)<\infty$ for all $k \geq 1$. The limit \eqref{ABSTIM} follows from the monotone convergence theorem

	\begin{eqnarray}
	\lim_{t \rightarrow \infty} t^\frexp P_i[T_{0,\frexp}>t]&=&\mu_1 \int_{ \theta^\star}^\infty\left(\lim_{t \rightarrow \infty} t^\frexp \ML{-\coef t^\frexp}\right)\frac{\Qt{i}}{\coef}d\Gamma(\coef)\\
	&=&\frac{\mu_1}{\Gamma(1-\frexp)}\int_{ \theta^\star}^\infty \frac{\Qt{i}}{\coef^2}d\Gamma(\coef).
	\end{eqnarray}	
The limit \eqref{ABSTIM2} is proved by using the same argument.
\end{proof}
The proposition \ref{limit1} leads us to the following theorem, interpreted as a Yaglom limit for the fractional case.
\begin{theorem}\label{QLD1}

Assume $\coef^\star>0$. For all $0<\frexp<1$ we have
\begin{equation}
\lim_{t \rightarrow \infty}{P_i[N_\frexp(t)=j|T_{0,\frexp}>t]}=\frac{\PP_{i,j}}{\sum_{n \geq 1} \PP_{i,j}}\label{YAGLIM3},
\end{equation}
where
\begin{eqnarray}\label{defPP}
\PP_{i,n}&=&\pi_n\left(\PP_{i,1}+\sum_{j=1}^{\min\{i,n-1\}}\frac{1}{\lambda_j \pi_j}\right), \hspace{1cm}n\geq 2,\nonumber\\
\PP_{i,1}&=&\frac{1}{\mu_1}.
\end{eqnarray}

\end{theorem} 

\begin{remark}
Clearly, the limit \ref{defPP} strongly depends on the initial condition $i$. We analyze in more detail the two extreme cases:
\begin{description}
	\item For $i=1$, the second term in the equation \eqref{defPP} vanishes, so
	\begin{equation}
	\lim_{t \rightarrow \infty}{P_i[N_\frexp(t)=j|T_{0,\frexp}>t]}=\frac{\pi_j}{\sum_{n \geq 1} \pi_n}\label{YAGLIM3sp}.
	\end{equation}
	\item For $i \rightarrow \infty$ (assuming that the limit exists) we get
	\begin{equation}
	\lim_{i \rightarrow \infty}\lim_{t \rightarrow \infty}{P_i[N_\frexp(t)=j|T_{0,\frexp}>t]}=\frac{\pi_n\left(\frac{1}{\mu_1}+\sum_{j=1}^{n-1}\frac{1}{\lambda_j \pi_j}\right)}{\sum_{n \geq 1}\pi_n\left(\frac{1}{\mu_1}+\sum_{j=1}^{n-1}\frac{1}{\lambda_j \pi_j}\right)}.
	\end{equation}
\end{description} 
The condition $\sum_{n \geq 1}\pi_n <\infty$ is equivalent to the almost sure absorption of the process, whereas the condition $\sum_{n \geq 1}\pi_n\sum_{j=1}^{n-1}\frac{1}{\lambda_j \pi_j}<\infty$ implies that the process comes down from infinity (according to  Theorem 3.1 of Van Doorn \cite{Doorn91} this   is equivalent to the existence of a unique quasi-stationary distribution).
It is interesting to notice that in the fractional model, the quasi-limiting behavior changes drastically compared to the Markovian case. Nevertheless, as we will see in the next section, the quasi-stationary distributions are the same in the fractional model. 
\end{remark}

\begin{proof}[Proof of Theorem \ref{QLD1}]

As a direct consequence of proposition \ref{limit1} we have
	\begin{equation}
\lim_{t \rightarrow \infty}{P_i[N_\frexp(t)=j|T_{0,\frexp}>t]}=\frac{\pi_j}{\mu_1}\frac{C_{i,j,1}}{C_{i,1,2}}\label{YAGLIM}.
\end{equation}
Since $p_{i,j,1}(t)=E_{i}[N_1(t)=j,T_0>t]$ , we get from the equation \eqref{transi0}
\begin{equation}\label{transi}
P_{i}[N_1(t)=j,T_0>t]=\pi_j \int_{\coef^\star}^{\infty}e^{-t \coef}\Qt{i}\Qt{j}d\Gamma(\coef).
\end{equation}
By taking the integral  over $t \geq 0$, the Fubini's Theorem yields to 
\begin{eqnarray}\label{lim1}
\int_{0 }^{\infty}P_{i}[N_1(t)=j,T_0>t]dt&=&\pi_j \int_{ 0}^{\infty}\int_{\coef^\star}^{\infty}e^{-t \coef}\Qt{i}\Qt{j}d\Gamma(\coef)dt\\
&=&\pi_j \int_{\coef^\star}^{\infty}\frac{\Qt{i}\Qt{j}}{\theta}d\Gamma(\coef),
\end{eqnarray}
analogously
\begin{eqnarray}\label{lim2}
\int_{0}^{\infty}P_{i}[T_0>t]dt&=&\mu_1 \int_{\coef^\star}^{\infty}\frac{\Qt{i}}{\theta^2}d\Gamma(\coef).
\end{eqnarray}
So that the limit \eqref{YAGLIM} is equivalent to
\begin{equation}
\lim_{t \rightarrow \infty}{P_i[N_\frexp(t)=j|T_{0,\frexp}>t]}=\frac{\int_{ 0}^{\infty}P_i[\mathds{1}_{N_1(t)=j,T_0>t}]dt}{\int_{ 0}^{\infty}P_i[\mathds{1}_{T_0>t}]dt}\label{YAGLIMalt}.
\end{equation}
 We recall the system of equations \eqref{FORW} for $\frexp=1$
\begin{eqnarray}\label{eqa1}
p^{\prime}_{i,j}(t)&=&\FORW,\hspace{1cm} j \geq 1.
\end{eqnarray}
By taking the  integral over $t \geq 0$, from the right hand of the  equation \eqref{eqa1} we obtain
\begin{eqnarray}
\int_{0}^{\infty}p^{\prime}_{i,j}(t) dt&=&\lim_{t \rightarrow \infty }p_{i,j}(t)-p_{i,j}(0)\\
&=&-\delta_{i,j}.
\end{eqnarray}
When $j=0$, from equation \eqref{FRW0} we get 
$$\int_{0}^{\infty }E_{i}[N_1(t)=1,T_0>t]=\frac{1}{\mu_1}.$$
By introducing   the notation
\begin{equation}
\Pij=\int_{0}^{\infty}E_{i}[N_1(t)=j,T_0>t]dt,
\end{equation}
with the convention $\PP_{i,0}=0$, we get the recurrence formula
\begin{eqnarray}
-\delta_{i,j}&=&\FORWL,
\end{eqnarray}
whose solution we can be computed explicitly. Let us re-arrange some terms
\begin{eqnarray}
-\delta_{i,j}&=&\FORWLB,
\end{eqnarray}
by taking the sum over $1 \leq j \leq n$
\begin{eqnarray}\label{eqint}
-\sum_{j=1}^n \delta_{i,j}&=&\lambda_0 P_{i,0}-\lambda_n \PP_{i,n}+\mu_{n+1}\PP_{i,n+1}-\mu_1P_{i,1},\hspace{1cm} j \geq 1.
\end{eqnarray}
Since $P_{i,0}=0$ and $P_{i,1}=\frac{1}{\mu_1}$, we get
\begin{eqnarray}\label{recur2}
1-\sum_{j=1}^n \delta_{i,j}&=&-\lambda_n \PP_{i,n}+\mu_{n+1}\PP_{i,n+1}.
\end{eqnarray}
We notice that
\begin{equation}
1-\sum_{j=1}^n \delta_{i,j}=\left\{ \begin{array}{ccc}
0& \textrm{if }n \geq i\\
1 &\textrm{if }n< i
\end{array}
\right.,
\end{equation}
the equation  \eqref{recur2} becomes
\begin{eqnarray}\label{recur3}
\mathds{1}_{n<i}&=&\mu_{n+1}\PP_{i,n+1}-\lambda_n \PP_{i,n}.
\end{eqnarray}
Recalling the identity $\frac{\mu_{n+1}}{\lambda_n \pi_n}=\frac{1}{\pi_{n+1}}$, we get now
\begin{eqnarray}\label{recur3b}
\frac{\mathds{1}_{n<i}}{\lambda_n \pi_n}&=&\frac{\PP_{i,n+1}}{\pi_{n+1}}-\frac{\PP_{i,n}}{\pi_{n}},
\end{eqnarray}
whose solution is
\begin{eqnarray}
\PP_{i,n}&=&\pi_n\left(\PP_{i,1}+\sum_{j=1}^{n-1}\frac{\mathds{1}_{j<i}}{\lambda_j \pi_j}\right), \hspace{1cm}n\geq 2,\\
\PP_{i,1}&=&\frac{1}{\mu_1}.
\end{eqnarray}
For $n \geq 2$, this is the same as
\begin{eqnarray}\label{defPP2}
\PP_{i,n}&=&\pi_n\left(\frac{1}{\mu_1}+\sum_{j=1}^{\min\{i,n-1\}}\frac{1}{\lambda_j \pi_j}\right),
\end{eqnarray}
so that the limit  is 
\begin{equation}
\lim_{t \rightarrow \infty}{P_i[N_\frexp(t)=j|T_{0,\frexp}>t]}=\frac{\PP_{i,j}}{\sum_{n \geq 1} \PP_{i,j}},
\end{equation}
concluding the proof.
\end{proof}

The following Theorem provides the convergence rate of the limit obtained in \eqref{YAGLIM}.

\begin{theorem}
	
	For all $0 < \frexp <1$, $\frexp \neq 1/2$ we have
	\begin{equation}
	\lim_{t \rightarrow \infty}	t^\frexp\left(\frac{\pijt}{P_i[T_{0,\frexp}>t]}-\QLD\right) =\QLD \frac{\Gamma(1-\frexp)}{\Gamma(1-2\frexp)}\left(\frac{C_{i,1,3}}{C_{i,1,2}}-\frac{C_{i,j,2}}{C_{i,j,1}}\right)\label{YAGLIM2},
	\end{equation}
	similarly for $\frexp=1/2$
	\begin{equation}
	\lim_{t \rightarrow \infty}	t^{}\left(\frac{\pijt}{P_i[T_{0,\frexp}>t]}-\QLD\right) =\frac{1}{2}\QLD\left(\frac{C_{i,1,4}}{C_{i,1,2}}-\frac{C_{i,j,3}}{C_{i,j,1}}\right)\label{YAGLIM2b}.
	\end{equation}
\end{theorem}
\begin{proof}
	For $\frexp\neq 1/2$ we  recall the asymptotic expansion of the Mittag-Leffler function for $t$ large enough (see equation \eqref{ASYM} from the appendix)
	\begin{eqnarray}
	\ML{-\coef t^\frexp }&=&\frac{1}{\Gamma(1-\frexp)}\frac{1}{\coef t^\frexp}-\frac{1}{\Gamma(1-2\frexp)}\frac{1}{\theta^2t^{2\frexp}}+O(t^{-3\frexp }).
	\end{eqnarray}
	The constants $C_{i,j,k}$ are finite, so that the following asymptotic expansions are valid 
	\begin{eqnarray}
	\pijt&=& \pi_j \left( \frac{C_{i,j,1}}{\Gamma(1-\frexp)}\frac{1}{t^\frexp}-\frac{C_{i,j,2}}{\Gamma(1-2\frexp)}\frac{1}{t^{2\frexp}}+o(t^{2\frexp})\right),\\
	P_i[T_{0,\frexp}>t]&=& \mu_1 \left( \frac{C_{i,1,2}}{\Gamma(1-\frexp)}\frac{1}{t^\frexp}-\frac{C_{i,1,3}}{\Gamma(1-2\frexp)}\frac{1}{t^{2\frexp}}+o(t^{2\frexp})\right),
	\end{eqnarray}
	and after some algebraic manipulations
	\begin{eqnarray}
	\frac{\pijt}{P_i[T_{0,\frexp}>t]}&=&\frac{\pi_j}{\mu_1}\frac{C_{i,j,1}}{C_{i,1,2}}\left( \frac{1-\frac{C_{i,j,2}}{C_{i,j,1}}\frac{\Gamma(1-\frexp)}{\Gamma(1-2\frexp)}\frac{1}{t^\frexp}+o(t^{-\frexp})}{1-\frac{C_{i,1,3}}{C_{i,1,2}}\frac{\Gamma(1-\frexp)}{\Gamma(1-2\frexp)}\frac{1}{t^\frexp}+o(t^{-\frexp})}\right).
	\end{eqnarray}	
	For $|z|$ small enough we know that $(1-z)^{-1}=1+z+o(z)$, so
	\begin{eqnarray}
	\frac{\pijt}{P_i[T_{0,\frexp}>t]}&=&\frac{\pi_j}{\mu_1}\frac{C_{i,j,1}}{C_{i,1,2}}\left( 1-\frac{1}{t^\frexp}\frac{\Gamma(1-\frexp)}{\Gamma(1-2\frexp)}\left(\frac{C_{i,j,2}}{C_{i,j,1}}- \frac{C_{i,1,3}}{C_{i,1,2}}\right)\right)+o(t^{-\frexp})
	\end{eqnarray}
	and consequently
	\begin{eqnarray}
	t^\frexp\left(\frac{\pijt}{P_i[T_{0,\frexp}>t]}-\frac{\pi_j}{\mu_1}\frac{C_{i,j,1}}{C_{i,1,2}}\right)&=&\frac{\pi_j}{\mu_1}\frac{C_{i,j,1}}{C_{i,1,2}}\left( \frac{\Gamma(1-\frexp)}{\Gamma(1-2\frexp)}\left(\frac{C_{i,1,3}}{C_{i,1,2}}-\frac{C_{i,j,2}}{C_{i,j,1}}\right)\right)+o(1).
	\end{eqnarray}
	The limit is obtained by letting $t 
	\rightarrow \infty$ and recalling the identity $\QLD=\QLDb$. For $\frexp=1/2$ we have to consider now the asymptotic expansion
	\begin{eqnarray}
	E_{1/2,1}{(-\coef t^{1/2}) }&=&\frac{1}{\coef t^{1/2}}\frac{1}{\Gamma(1/2)}+\frac{1}{\theta^3t^{3/2}}\frac{1}{\Gamma(-1/2)}+O(t^{-2 }),
	\end{eqnarray}
	similarly
	\begin{eqnarray}
	\frac{\pijt}{P_i[T_{0,\frexp}>t]}&=&\frac{\pi_j}{\mu_1}\frac{C_{i,j,1}}{C_{i,1,2}}\left( 1+\frac{1}{t}\frac{\Gamma(1/2)}{\Gamma(-1/2)}\left(\frac{C_{i,j,3}}{C_{i,j,1}}- \frac{C_{i,1,4}}{C_{i,1,2}}\right)\right)+o(t^{-1}),
	\end{eqnarray}
	since  $\frac{\Gamma(1/2)}{\Gamma(-1/2)}=-\frac{1}{2}$ the proof concludes by using the same argument as the  case $\frexp \neq 1/2$.
\end{proof}
\section{Quasi Stationary Distributions}\label{QSDB}

In this section, we suppose that the initial state  $N_\frexp(0)$ is random, with a distribution $\qsd$ supported on $\N^+=\{1,2,3,\cdots\}$. In this case,
\begin{eqnarray}\label{QSDIS}
P_\qsd[N_\frexp(t)=k]&=& \sum_{i \geq 1} P[N_\frexp(0)=i]P[N_\frexp(t)=k|N_\frexp(0)=i], \\
&=&\sum_{i \geq 1} \qsd_i\pijt,
\end{eqnarray}
where $\pijt$ are the transition probabilities defined in the previous sections. More generally, given $A \subseteq \N^+$ we write
\begin{equation}\label{QSDISB}
P_\qsd[N_\frexp(t) \in A]=\sum_{k \in A} P_\qsd[N_\frexp(t)=k].
\end{equation}
We notice  that 
\begin{equation}\label{QSDISB2}
\pjnu= P_\qsd[N_\frexp(t)=j]
\end{equation}
is the solution to the system of equations 
\begin{eqnarray}\label{FORWsp}
\derfra p_{j,\frexp}(t) &=&\FORWsp, \hspace{0.5cm} j \geq 2,\\
\derfra p_{1,\frexp}(t)&=&-(\lambda_1+\mu_1)p_{1,\frexp}(t)+\mu_2p_{2,\frexp}(t), \nonumber
\end{eqnarray} 
with initial condition $p_{j,\frexp}(0)=\nu_j$, $j\geq 1$.

We say that a probability measure $\qsd$ is a quasi stationary distribution if for all  $A \subseteq \N^+$ and $t \geq 0$ the following identity follows
\begin{equation}
P_\qsd[N_\frexp(t) \in A, T_{0,\frexp} >t]=\qsd(A)P_\qsd[T_{0,\frexp} >t].
\end{equation}
When $\frexp=1$, a  quasi-stationary stationary distribution is a solution to the system $\qsd^t Q^{(a)}=-\theta \qsd$
for some $\theta \in (0, \theta^*]$, where $-Q^{(a)}$ is the matrix obtained by removing from the original matrix $Q$ the row and the column associated with the absorbent state $0$. Moreover
\begin{eqnarray}
P_\qsd[T_{0}>t]&=&e^{-\theta t}\label{exp1}\\
P_\qsd[N_1(t)=j,T_{0}>t]&=&\qsd_j e^{-\theta t} \label{exp2}.
\end{eqnarray}
The next proposition states a similar property in the fractional case.
\begin{proposition}\label{QSDPROP}
	Let $\nu$ be a quasi stationary distribution, then for all $t>0$ the following identities are satisfied
	\begin{eqnarray}
	P_\qsd[T_{0,\frexp}>t]&=&\ML{-\theta t^\frexp}\label{abs1}\\
	P_\qsd[N_\frexp(t)=j, T_{0,\frexp}>t]&=&\qsd_j \ML{-\theta t^\frexp}\label{abs2}.
	\end{eqnarray}
	\begin{proof}
		Is a direct consequence of Theorem \ref{ChanTime}. Given $t \geq 0$ fixed, from equation \eqref{exp1} we have
		\begin{eqnarray}
		P_\qsd[T_{0,\frexp}>t]&=&\int_{0}^{\infty}P_\qsd[T_{0}>u]\dPE\\
		&=&\int_{0}^{\infty}e^{-\theta u}\dPE\\
		&=&\ML{-\coef t^\frexp}.
		\end{eqnarray}
		The equation \eqref{abs2} is obtained from \eqref{exp2} by using the same argument.
	\end{proof}
\end{proposition}

The following theorem states that the family of quasi-stationary distributions is the same for all $\frexp \in (0,1]$.
\begin{theorem}\label{CaracQSD}
	A probability measure $\qsd$ is quasi-stationary  distribution  if and only if solves  the system of equations
	\begin{eqnarray}\label{sistqsd}
	-\theta \qsd_j&=&\FORWqsd, \hspace{1cm}j \geq 2,\\
	-\theta \qsd_1&=&-(\lambda_1+\mu_1)\qsd_1+\mu_2 \qsd_2,
	\end{eqnarray} 
where $\theta=\mu_1 \qsd_1$. 
\end{theorem}
\begin{proof}
We follow the same approach as Van Doorn \cite{Doorn91}. If $\nu$ is a probability measure that  solves  \eqref{sistqsd} for $\theta=\mu_1 \qsd_{1}$, the transition probabilities defined as
	\begin{eqnarray}
	p_{0,\frexp}(t)&=&1-\ML{-\theta t^\frexp},\\
	p_{j,\frexp}(t)	&=& \qsd_j \ML{-\theta t^\frexp},
	\end{eqnarray}
	satisfy the system of equations \eqref{FORWsp} with initial distribution $p_{j,\frexp}(0)= \qsd_j$, $j \geq 1$. Reciprocally, if $\qsd$ is a quasi-stationary distribution, we know from proposition \ref{QSDPROP}  that $\pjnu(t)$ defined in \eqref{QSDISB2} satisfies
	\begin{eqnarray}
\mathbb{D}^{\frexp}\pjnu&=&-\coef \pjnu.
	\end{eqnarray}
In addition, $\pjnu$ is a solution to \eqref{FORWsp} $p_{j,\frexp}(0)=\qsd_j$, $j \geq 1$ and consequently
	 \begin{eqnarray}
-\coef \pjnu(t)&=&\lambda_{j-1}  p_{j-1,\frexp}(t)-(\lambda_j+\mu_j) p_{j,\frexp}(t) +\mu_{j+1} p_{j+1,\frexp}(t), ~~j \geq 2,\\
-\coef  p_{1,\frexp}(t)&=& -(\lambda_1+\mu_1) p_{1,\frexp}(t) +\mu_{2} p_{2,\frexp}(t).
	 \end{eqnarray}
By taking the limit $t \rightarrow 0^+$  we deduce that $\nu$ is a solution to \eqref{sistqsd}. Finally, from equation \eqref{abs1} we get 
\begin{eqnarray}
\mathbb{D}^{\frexp} p_{0,\frexp}(t)&=& \mathbb{D}^{\frexp} (1-P_\qsd[T_{0,\frexp}>t])\\
&=& \coef P_\qsd[T_{0,\frexp}>t],
\end{eqnarray}
similarly
	\begin{equation}
\mathbb{D}^{\frexp} p_{0,\frexp}(t)= \mu_1  p_{1,\frexp}(t), 
	\end{equation}
so $ \mu_1  p_{1,\frexp}(t)=\coef P_\qsd[T_{0,\frexp}>t]$. By letting $t \rightarrow 0^+$ we deduce $\coef=\mu_1 \nu_1 $, concluding the proof.
\end{proof}
Since the family of quasi-stationary distributions is the same for the complete interval $\frexp \in (0,1]$, its characterization coincides with the one originally presented by Van Doorn, enunciated below. 
\begin{theorem}
	[van Doorn \cite{Doorn91}] 
	If the series $D=\sum_{i\geq 2}(\mu_{i}\pi_{i})^{-1}\sum_{j\geq i}\pi_{j}$  diverges, then either $\theta^\star=0$ and there is no qsd distribution or $\theta^\star>0$, in which case there is a family of qsd distributions indexed by $\qsd_\theta$, $\theta\in (0,\theta^\star]$. If the $D$ series converges, then there is a unique distribution qsd indexed by $\theta^\star$.
\end{theorem}

However, the theorem \ref{QLD1} implies they do not necessarily attract the initial distributions, unless that the initial distribution coincides with some of the qsd. Also, the extinction rate now decays proportionally to $t^\frexp$. This behavior, which turns out to be different from the case $\frexp=1$, is a direct consequence of the long memory nature of the process.
\section{The Linear Process}\label{EXAM}
We revisit the linear model, previously studied by Orsingher \& Polito \cite{LinearBDP2011}. The birth rates and the death rates are $\lambda_{i}=i \lambda $ and $\mu_{i}=i \mu$ respectively. To make our analysis simpler, we assume first that the initial condition is $N_\frexp(0)=1$. It is well known that in the case $\frexp=1$, the transition probabilities  are
\begin{equation}
p_{1,j,1}(t)=\left\{\begin{array}{ll}\frac{(\lambda t)^{k-1}}{(1+\lambda t)^{k+1}} & {\lambda=\mu} ,\\  \frac{\left(\lambda\left(1-e^{-(\lambda-\mu) t}\right)\right)^{j-1}}{\left(\lambda-\mu e^{-(\lambda-\mu) t}\right)^{j+1}} (\lambda-\mu)^{2}  e^{-(\lambda-\mu) t}& {\lambda \neq \mu}.\end{array}\right.
\end{equation}
Similarly, the probabilities of non extinction are
\begin{equation}\label{PRNE}
P_1[T_{0}>t]=\left\{\begin{array}{ll}\frac{\lambda-\mu}{\lambda}\left(1- \sum_{m \geq 1}\left(\frac{\mu}{\lambda}\right)^{m} e^{-(\lambda-\mu) m t }\right)& {\lambda>\mu}, \\ {\frac{\mu-\lambda}{\lambda} \sum_{m \geq 1}\left(\frac{\lambda}{\mu}\right)^{m}  e^{-(\mu-\lambda) m t }} & {\lambda<\mu}, \\\frac{1}{1+\lambda t} & {\lambda=\mu}.\end{array}\right.
\end{equation}
A widely known fact is that the asymptotic behavior  depends on  the ratio $\lambda/\mu$. In the fractional case the same occurs, so we study the three cases  separately.

	\subsection {The case $\lambda<\mu$}
    As we mentioned, we study  first the asymptotic behavior of the process with initial state $N_\frexp(0)=1$. When $\frexp \in (0,1)$ the probability of non extinction is
	\begin{equation}
	P_1[T_{0,\frexp}>t]=\left( \frac{\mu-\lambda}{\lambda}\right)\sum_{m \geq 1}({\lambda}/{\mu})^m \ML{-(\mu-\lambda)mt^\frexp}.
	\end{equation}
We recall  that for all $m\geq 1$ we have the limit
	\begin{eqnarray}\label{LIMI}
\lim_{t \rightarrow \infty }	t^\frexp  \ML{-(\mu-\lambda)mt^\frexp}=\frac{1}{\Gamma(1-\frexp)} \frac{1}{(\mu-\lambda)m}.
	\end{eqnarray}
From \ref{LIMI}  and the dominated convergence theorem, we obtain
\begin{eqnarray}\label{ABSTIME}
\lim_{t \rightarrow \infty } t^\frexp P_1[T_{0,\frexp}>t]&=&\left( \frac{\mu-\lambda}{\lambda}\right) \lim_{t \rightarrow \infty } t^\frexp  \sum_{m \geq 1}({\lambda}/{\mu})^m \ML{-(\mu-\lambda)mt^\frexp}\\
&=&\left( \frac{\mu-\lambda}{\lambda}\right)  \sum_{m \geq 1}({\lambda}/{\mu})^m \lim_{t \rightarrow \infty } t^\frexp  \ML{-(\mu-\lambda)mt^\frexp}\\
&=&\frac{1}{\Gamma(1-\frexp)}\frac{1}{ \lambda }\sum_{m \geq 1}\frac{(\lambda/\mu)^m}{m}\\
&=&-\frac{1}{\Gamma(1-\frexp)}\frac{1}{\lambda}\ln \left( 1-\frac{\lambda}{\mu}\right).
\end{eqnarray}
The limit can be deduced in an alternative way. We now consider the representation
$$	P_1[T_{0,\frexp}>t]=\frac{\mu-\lambda}{\lambda}\frac{(\lambda/\mu)e^{-(\mu-\lambda)t}}{1-\frac{\lambda}{\mu}e^{-(\mu-\lambda)t}},$$ 
from equation \eqref{lim2}
\begin{eqnarray}\label{ABSTIMEalt}
\lim_{t \rightarrow \infty } P_1[T_{0,\frexp}>t]&=& \frac{1}{\Gamma(1-\frexp)}\int_{0}^{\infty}\frac{\mu-\lambda}{\lambda}\frac{(\lambda/\mu)e^{-(\mu-\lambda)t}}{1-\frac{\lambda}{\mu}e^{-(\mu-\lambda)t}}du \\
&=&\frac{1}{\Gamma(1-\frexp)}\left.\ln \left( 1-\frac{\lambda}{\mu}e^{-(\mu-\lambda)t}\right)\right|_{0}^{\infty}\\
&=&-\frac{1}{\Gamma(1-\frexp)}\frac{1}{\lambda}\ln \left( 1-\frac{\lambda}{\mu}\right).
\end{eqnarray}
Similarly, for all $j \geq 1$
\begin{eqnarray}
\int_{0 }^{\infty}P_{i}[N_1(t)=j,T_0>t]dt&=&(\lambda-\mu)^{2} \int_{0}^{\infty} \frac{\left(\lambda\left(1-e^{-(\lambda-\mu) t}\right)\right)^{j-1}}{\left(\lambda-\mu e^{-(\lambda-\mu) t}\right)^{j+1}}   e^{-(\lambda-\mu) t}dt\\
&=& (\lambda-\mu)^{2} \int_{0}^{\infty} \frac{\left(\lambda\left(e^{(\lambda-\mu) t}-1\right)\right)^{j-1}}{\left(\lambda e^{(\lambda-\mu) t}-\mu\right)^{j+1}}   e^{(\lambda-\mu) t}dt
\end{eqnarray}
and after some computations we get
\begin{eqnarray}\label{limitb}
\int_{0 }^{\infty}P_{1}[N_1(t)=j,T_0>t]dt&=&\frac{1}{\lambda }\frac{(\lambda/\mu)^j}{j}.
\end{eqnarray}
Conditioned to $N_\frexp(0)=1$, the quasi-limiting distribution is deduced from  \eqref{ABSTIMEalt} and \eqref{limitb}
	\begin{equation}
\lim_{t \rightarrow \infty}{P_i[N_\frexp(t)=j|T_{0,\frexp}>t]}=\frac{\frac{(\lambda/\mu)^j}{j}}{\sum_{j \geq 1}\frac{(\lambda/\mu)^j}{j}}\label{YAGLIM3spex1}.
\end{equation}
Since $P_{1,j}=\frac{(\lambda/\mu)^{j-1}}{j}$, it can also be written into the form
\begin{eqnarray}
\lim_{t \rightarrow \infty}{P_1[N_\frexp(t)=j|T_{0,\frexp}>t]}&=&\frac{P_{1,j}}{\sum_{j\geq 1} P_{1,j}},
\end{eqnarray}
which is consistent with our theorems. Now, the quasi limiting behavior when the initial condition is $N_\frexp(0) > 1$, can be obtained directly from \eqref{YAGLIM3}. We recall the expressions $\pi_j=\frac{(\lambda/\mu)^{j-1}}{j}$ and  $\lambda_j=j\lambda$, so  $\lambda_j \pi_j=\mu (\lambda/\mu)^j$ and equation \eqref{defPP2} becomes
\begin{eqnarray}
P_{i,j}&=&\frac{(\lambda/\mu)^{j-1}}{j}\left[ \frac{1}{\mu}+\sum_{k=1}^{\min\{ i,j-1\}}\frac{1}{\mu (\lambda/\mu)^k}\right]\\
&=&\frac{(\lambda/\mu)^{j}}{\lambda j} \sum_{k=0}^{\min\{ i,j-1\}}\frac{1}{ (\lambda/\mu)^k}\\
&=&\frac{(\lambda/\mu)^{j}}{\lambda j}  \frac{\CUO^{-1-\min\{ i,j-1\}}-1}{\CUO^{-1}-1}\label{QLDLI}.
\end{eqnarray}
The equation \ref{QLDLI} is the same as
\begin{equation}
P_{i,j}= \left\{ \begin{array}{ll}
\frac{1}{\lambda j}  \frac{1-\CUO^{j}}{\CUO^{-1}-1}, &     j \leq i ,\\
\frac{1}{\lambda j}  \frac{\CUO^{j-i-1}-\CUO^{j}}{\CUO^{-1}-1}, &  j \geq i+1. \\
\end{array}
\right.
\end{equation}
The quasi limiting distribution is
\begin{eqnarray}
\frac{P_{i,j}}{\sum_{j \geq 1}P_{i,j}}&=&  \frac{\frac{1}{j}(\CUO^{-1+\max\{ 1,j-i\}}-\CUO^j)}{\sum_{j \geq 1}\frac{1}{j}(\CUO^{-1+\max\{ 1,j-i\}}-\CUO^j)}.
\end{eqnarray}
We notice that in the limit case $i \rightarrow \infty$ we have $P_{\infty,j}=\frac{1}{\lambda j}  \frac{1-\CUO^{j}}{\CUO^{-1}-1} $, which is not a probability measure since  $\sum_{j \geq 1}\frac{1}{j}=\infty$.
\subsection{The case $\lambda \geq \mu$} 
We study first the case $\lambda =\mu$. The probability of non extinction is
 \begin{eqnarray}
 P_1[T_{0,\frexp}>t]=\int_0^\infty e^{-z}\ML{-\lambda t^\frexp z}dz.
 \end{eqnarray}
 Alternatively, for $\frexp=1$ we have  the identity  $P[T_{0}>t]=\frac{1}{1+\lambda t}$. From  theorem \ref{ChanTime} we get the formula
  \begin{eqnarray}
 P_1[T_{0,\frexp}>t]&=&E\left[\frac{1}{1+\lambda \Et{t}}\right].
 \end{eqnarray}
 In order to study the asymptotic behavior, we introduce the function
$$
f(t)= t^{\frexp} \left(\log \log t^\frexp \right)^{1-\frexp}.
 $$
 It is  well known that (see for instance Bertoin \cite{Bertoin99})
 \begin{equation}\label{iterlog}
 \limsup_{t \rightarrow \infty}\frac{\Et{t}}{f(t)}=C_\frexp \hspace{1cm}
 \end{equation}
with probability $1$, for some positive constant $C_\frexp$ depending only on $\frexp$ . By taking the limit $t \rightarrow \infty$, we now get
   \begin{eqnarray}
 \liminf_{t \rightarrow \infty} f(t) P_1[T_{0,\frexp}>t]
 &=&\liminf_{t \rightarrow \infty} E\left[\frac{f(t)}{1+\lambda   \Et{t} }\right]\\
 &\geq& E\left[\liminf_{t \rightarrow \infty} \frac{f(t)}{1+\lambda   \Et{t} }\right]\\
 &\geq& E\left[ \frac{1}{\limsup_{t \rightarrow \infty} \frac{1}{f(t)}+\lambda \frac{ \Et{t}}{f(t)} }\right].
 \end{eqnarray}
From equation \eqref{iterlog} and noticing that $\lim_{t \rightarrow \infty} f(t)=\infty$ we conclude
  \begin{eqnarray}
 \liminf_{t \rightarrow \infty} f(t) P_1[T_{0,\frexp}>t]
 &\geq& \frac{1}{\lambda C_\frexp}.
 \end{eqnarray}
Similarly, we compute  

\begin{eqnarray}
\limsup_{t \rightarrow \infty } \frac{\pjt}{ P_1[T_{0,\frexp}>t]}&=&\limsup_{t \rightarrow \infty } \frac{f(t)\pjt}{ f(t)P_1[T_{0,\frexp}>t]}\\
&\leq& \frac{ \limsup_{t \rightarrow \infty }f(t)\pjt}{  \liminf_{t \rightarrow \infty }f(t)P_1[T_{0,\frexp}>t]}\\
&\leq& \lambda C_\frexp  \limsup_{t \rightarrow \infty }f(t)\pjt.
\end{eqnarray}
The identity  $\pjt=E\left[ \frac{(\lambda \Et{t})^{j-1}}{(1+\lambda \Et{t})^{j+1}}\right]$ allows us to compute the limit
\begin{eqnarray}
\limsup_{t \rightarrow \infty } f^2(t)\pjt &=&\limsup_{t \rightarrow \infty }f^{2}(t)E\left[ \frac{(\lambda \Et{t})^{j-1}}{(1+\lambda \Et{t})^{j+1}}\right]\\
&\leq&E\left[ \limsup_{t \rightarrow \infty } f^{2}(t)\frac{(\lambda \Et{t})^{j-1}}{(1+\lambda \Et{t})^{j+1}}\right]\\
&\leq&\frac{1}{\lambda^2}E\left[ \limsup_{t \rightarrow \infty }\frac{1}{(\Et{t}/f(t))^2}\right]\\
&=&\frac{1}{(\lambda C_\frexp)^2}.
\end{eqnarray}
Concluding 
\begin{eqnarray}
 \limsup_{t \rightarrow \infty }f(t)\pjt&\leq& \lim_{t \rightarrow \infty } \frac{1}{f(t)}\limsup_{t \rightarrow \infty } f^2(t)\pjt\\
 &=&0.
\end{eqnarray}
Consequently, there is no a quasi limiting distribution as expected. Finally, when $\lambda>\mu$ the  probability of non extinction is
\begin{eqnarray}
P_1[T_{0,\frexp}>t]=\frac{\mu}{\lambda}-\frac{\lambda-\mu}{\lambda} \sum_{m \geq 1}\left(\frac{\mu}{\lambda}\right)^{m} \ML{-(\lambda-\mu) mt^{\frexp}},
\end{eqnarray}
 which is a  strictly positive value, so that there is no a quasi-limiting distribution. By following a similar argument as the previous case, we get the  limit
\begin{equation}
\lim_{t \rightarrow \infty } t^\frexp \left(P_1[T_{0,\frexp}>t]-\frac{\lambda-\mu}{\lambda}\right)=\frac{1}{\Gamma(1-\frexp)}\frac{\lambda-\mu}{ \lambda }\sum_{m \geq 1}\frac{\CUO^m}{m}.
\end{equation}
\appendix
\section{The Mittag-Leffler function}\label{PROPML}
 
We present a  summary concerning  the basic properties of the Mittag-Leffler function. 
\begin{definition}
The complex-valued Mittag-Leffler function with one-parameter $\frexp \in (0,1]$ is defined as
	\begin{equation}
	\ML{z}=\defML{z},\hspace{1cm} z \in \C.
	\end{equation}
In particular, when $\frexp=1$ we have that  $E_{1,1}(z)=e^z$ is the exponential function.
\end{definition}
In the real valued case, it is well known that (see Proposition 3.23 \cite{Gorenflo2014}) the Mittag-Leffler function with negative argument $\ML{-x}$, $x>0$ is completely monotonic for all $0 \leq \frexp \leq 1$, \textrm{i.e.}
	\begin{equation}
	(-1)^n\frac{d^n}{d x^n }(\ML{-x}) \geq 0.
	\end{equation}
This is equivalent to the existence of a representation of $\ML{-x}$ in the form of a Laplace-Stieljes integral with non decreasing density and non-negative measure $d\mu$
$$
\ML{-x}=\int_{0}^{\infty}e^{-xu}d\mu(u).
$$
Since $\ML{0}=1$ we have from the dominated convergence theorem that $\mu$ is a probability measure. In fact, we know from equation \eqref{LEBSTI} for $s=x$ and $t=1$
\begin{eqnarray}
\ML{-x}&=&E[e^{-x\Et{1}}]\\
&=&\int_{0}^{\infty}e^{-xu}\dPEc,
\end{eqnarray}
and consequently the measure is $d\mu(u)=\dPEc$. In particular, from this integral representation we get that  the function $\ML{x}$ is strictly positive for $x \geq 0$.
\begin{proposition}
For all $\lambda>0$, $f(x)=E_{\frexp}(-\lambda x^\frexp)$ is  the unique solution to the  equation
	\begin{equation}\label{PVIFR}
	\mathbb{D}^\frexp f(x)=-\lambda f(x),\hspace{1cm}f(0)=1.
	\end{equation}
	\begin{proof}
It proceeds by direct computation by taking  the Laplace transform
\begin{equation}
\LAPL{f}=\int_{0}^{\infty} e^{-sx}f(x)dx
\end{equation}
at both sides of the equation \eqref{PVIFR}. At the right hand we just have $\LAPL{-\lambda f}=-\lambda\LAPL{f}$, whereas at left hand we have
		\begin{eqnarray}
		\LAPL{\mathbb{D}^{\frexp}f}&=&\LAPL{\int_0^x\frac{f^\prime(u) }{(x-u)^{\frexp}}du}\\
		&=&\LAPL{f^\prime}\LAPL{\frac{t^{-\frexp}}{\Gamma(1-\frexp)}}\\
		&=&\left( s\LAPL{f}-f(0)\right)s^{\frexp}.
		\end{eqnarray}
	This implies 
		\begin{equation}
		\LAPL{f}=\frac{s^{\frexp-1}}{\lambda+s^\frexp},
		\end{equation}
by taking the inverse 
\begin{equation}
f(x)=\defML{(-\lambda x^\frexp)}
\end{equation}
we conclude  $f(x)=E_{\frexp,1}(-\lambda x^\frexp)$.
	\end{proof}
\end{proposition}
Concerning the asymptotic behavior of the complex valued Mittag-Leffler function, we take the formulas formulas 4.7.4 and 4.7.5 from p-75 of \cite{Gorenflo2014}. Let be  $\frexp \in (0,2)$ and let be $\frac{\pi \frexp}{2}<\theta<\min\{\pi,\frexp \pi\}$, where $\theta=arg(z)$. For all $|z|$ large enough we have
	\begin{eqnarray}
	{E_{\frexp}}(z)&=&\frac{1}{\frexp}{e^{z^{1/\frexp}}}-\sum_{m=1}^{N-1}\frac{z^{-m}}{\Gamma(1-m\frexp)}+O(|z|^{-N}),\hspace{0.5cm}|\arg(z)|\leq \theta.\\
	{E_{\frexp}}(z)&=&-\sum_{m=1}^{N-1}\frac{z^{-m}}{\Gamma(1-m\frexp)}+O(|z|^{-N}),\hspace{0.5cm}\theta \leq |\arg(z)|\leq \pi.
	\end{eqnarray}
In particular, when $\arg(z)=0$ and $\arg(z)=\pi$ we recover the asymptotic expansion for the real valued case. In fact, when $|x|$ is large enough we get the following approximations
\begin{eqnarray}
\ML{x}&=&\frac{1}{\frexp}{e^{x^{1/\frexp}}}-\frac{1}{x}\frac{1}{\Gamma(1-\frexp)}-\frac{1}{x^2}\frac{1}{\Gamma(1-2\frexp)}+O(x^{-3}),\hspace{0.5cm}x>0,\\
\ML{-x}&=&\frac{1}{x}\frac{1}{\Gamma(1-\frexp)}-\frac{1}{x^2}\frac{1}{\Gamma(1-2\frexp)}+O(x^{-3}),\hspace{0.5cm}x>0\label{ASYM}.
\end{eqnarray}
The variable $x$ can be replaced by $\pm \lambda t^{\frexp}$, so 
\begin{eqnarray}
\ML{\lambda t^\frexp}&=&\frac{1}{\frexp}{e^{\lambda ^{1/\frexp}t}}-\frac{1}{\lambda t^\frexp }\frac{1}{\Gamma(1-\frexp)}-\frac{1}{\lambda^2 t^{2 \frexp}}\frac{1}{\Gamma(1-2\frexp)}+O(t^{-3\frexp }),\hspace{0.5cm}t>0,\\
\ML{-\lambda t^\frexp}&=&\frac{1}{\lambda t^\frexp }\frac{1}{\Gamma(1-\frexp)}-\frac{1}{\lambda^2t^{2 \frexp}}\frac{1}{\Gamma(1-2\frexp)}+O(t^{-3\frexp }),\hspace{0.5cm}t>0\label{ASYMa}.
\end{eqnarray}
Moreover, the following limits are valid for all $\lambda>0$.
\begin{eqnarray}
\lim_{t \rightarrow \infty} t^\frexp{E_{\frexp,1}}(-\lambda t^\frexp)&=&\frac{1}{\lambda}\frac{1}{\Gamma(1-\frexp)},\\
\lim_{t \rightarrow \infty} {e^{-\lambda^{1/\frexp} t}}{E_{\frexp,1}}(\lambda t^\frexp)&=&\frac{1}{\frexp}.
\end{eqnarray}
\section*{Acknowledgements}
J.L. acknowledges the financial  support  of the grant program FONDECYT de Iniciaci\'on en Investigaci\'on, Project No. 11140479  and to the following Mathematics Institutes where part of this work were done: Universidad Cat\'olica del Norte (Antofagasta, Chile) and Instituto Potosino de Investigaci\'on Cient\'ifica y Tecnol\'ogica (San Luis Potos\'i, M\'exico). 
\newpage
\bibliographystyle{elsarticle-num}
\bibliography{Kernel}

\end{document}

%% file: Commands.tex
\newcommand{\CUO}{(\lambda/\mu)}
\newcommand{\frexp}{\alpha}
\newcommand{\Tau}{\mathcal{T}}
\newcommand{\Et}[1]{L({#1})}
\newcommand{\Dt}[1]{D(#1)}

\newcommand{\rate}[1]{\lambda_{#1}+\mu_{#1}}
\newcommand{\pijt}{p_{i,j,\frexp}(t)}

\newcommand{\pjnu}{p_{ j,\frexp}}
\newcommand{\coef}{\theta}
\newcommand{\dPE}{h_\frexp(u,t)du}
\newcommand{\dPEc}{h_\frexp(u,1)du}
\newcommand{\dPEb}{h_\frexp(u,t)}

\newcommand{\C}{\mathbb{C}}
\newcommand{\N}{\mathbb{N}}
\newcommand{\Qt}[1]{\psi_{#1}(\coef)}
\newcommand{\SUMT}[1]{\cS_{n,\frexp}}

\newcommand{\pojt}{p_{0,j,\frexp}(t)}
\newcommand{\piot}{p_{i,0,\frexp}(t)}
\newcommand{\defpijt}{P[N_{\frexp}(t)=j|N_{\frexp}(0)=i]}

\newcommand{\Fracdef}[1]{\frac{1}{\Gamma(1-#1)}\int_0^t \frac{f^{\prime}(s)}{(t-s)^#1}ds}

\newcommand{\derFr}[1]{\mathbb{D}^#1 f}

\newcommand{\derfra}[1]{\mathbb{D}^\frexp #1}

\newcommand{\derfrM}[1]{\mathbb{D}^#1 P(t)}

\newcommand{\BACK}{\mu_ip_{i-1,j,\frexp}(t)-(\lambda_i+\mu_i)p_{i,j,\frexp}(t)+\lambda_i p_{i+1,j,\frexp}(t)}
\newcommand{\BACKLAP}{\mu_i\Phi_{i-1,j}(s^\frexp)-(\lambda_i+\mu_i)\Phi_{i,j}(s^\frexp)+\lambda_i \Phi_{i+1,j}(s^\frexp)}
\newcommand{\BACKLAPB}{\mu_i\Psi_{i-1,j}(s)-(\lambda_i+\mu_i)\Psi_{i,j}(s)+\lambda_i \Psi_{i+1,j}(s)}

\newcommand{\FORW}{\lambda_{j-1}p_{i,j-1,\frexp}(t)-(\lambda_j+\mu_j)p_{i,j,\frexp}(t)+\mu_{j+1} p_{i,j+1,\frexp}(t)}
\newcommand{\FORWL}{\lambda_{j-1}\PP_{i,j-1}-(\lambda_j+\mu_j)\PP_{i,j}+\mu_{j+1} \PP_{i,j+1}}
\newcommand{\FORWLB}{\left(\lambda_{j-1}\PP_{i,j-1}-\lambda_j\PP_{i,j} \right)+\left(\mu_{j+1} \PP_{i,j+1}-\mu_j\PP_{i,j}\right)}
\newcommand{\FORWsp}{\lambda_{j-1}p_{j-1,\frexp}(t)-(\lambda_j+\mu_j)p_{j,\frexp}(t)+\mu_{j+1} p_{j+1,\frexp}(t)}

\newcommand{\FORWqsd}{\lambda_{j-1}\qsd_{j-1}-(\lambda_j+\mu_j)\qsd_{j}+\mu_{j+1} \qsd_{j+1}}

\newcommand{\LAPL}[1]{\mathcal{L}\left[#1\right](s)}

\newcommand{\ML}[1]{E_{\frexp,1}(#1)}

\newcommand{\defML}[1]{\sum_{k \geq 0}\frac{#1^k}{\Gamma(\frexp k+1)}}

\newcommand{\qsd}{\nu}
\newcommand{\PP}{P}
\newcommand{\pjt}{p_{1,j,\frexp}(t)}
\newcommand{\Pij}{\PP_{i,j}}
\newcommand{\QLD}{\frac{\PP_{i,j}}{\sum_{n \geq 1} \PP_{i,j}}}
\newcommand{\QLDb}{\frac{\pi_j}{\mu_1}\frac{C_{i,j,1}}{C_{i,1,2}}}

\newcommand{\cS}{\mathcal{S}}

\newtheorem{remark}{Remark}

\newtheorem{theorem}{Theorem}

\newtheorem{definition}{Definition}
\newtheorem{proposition}{Proposition}